\documentclass[11pt]{article}
\usepackage[a4paper,top=2.6cm,bottom=2.6cm,left=2.7cm,right=2.7cm]{geometry}
\usepackage{amsmath,amssymb,amsthm,mathtools}
\usepackage{microtype}
\usepackage{enumitem}
\usepackage[hidelinks]{hyperref}
\usepackage{newunicodechar}
\newcommand{\St}{\operatorname{St}}

\newcommand{\Aut}{\operatorname{Aut}}
\newcommand{\Tor}{\operatorname{Tor}}
\newcommand{\reg}{\operatorname{reg}}
\newcommand{\depth}{\operatorname{depth}}
\newcommand{\pd}{\operatorname{pd}}

\newcommand{\NS}{\operatorname{NS}}
\newcommand{\qbinom}[2]{\genfrac{[}{]}{0pt}{}{#1}{#2}_p}
\newtheorem{theorem}{Theorem}[section]
\newtheorem{proposition}[theorem]{Proposition}
\newtheorem{corollary}[theorem]{Corollary}
\newtheorem{lemma}[theorem]{Lemma}
\theoremstyle{remark}
\newtheorem{remark}[theorem]{Remark}
\newunicodechar{Φ}{\ensuremath{\Phi}}
\newunicodechar{Δ}{\ensuremath{\Delta}}
\newunicodechar{β}{\ensuremath{\beta}}
\newunicodechar{μ}{\ensuremath{\mu}}
\newunicodechar{ρ}{\ensuremath{\rho}}
\newunicodechar{α}{\ensuremath{\alpha}}
\newunicodechar{Σ}{\ensuremath{\sum}}
\newunicodechar{≃}{\ensuremath{\simeq}}
\newunicodechar{≅}{\ensuremath{\cong}}
\newunicodechar{≤}{\ensuremath{\le}}
\newunicodechar{≥}{\ensuremath{\ge}}
\newunicodechar{≠}{\ensuremath{\ne}}
\newunicodechar{∈}{\ensuremath{\in}}
\newunicodechar{⊆}{\ensuremath{\subseteq}}
\newunicodechar{∩}{\ensuremath{\cap}}
\newunicodechar{∪}{\ensuremath{\cup}}
\newunicodechar{→}{\ensuremath{\to}}
\newunicodechar{↦}{\ensuremath{\mapsto}}
\newunicodechar{∨}{\ensuremath{\vee}}
\newunicodechar{−}{\ensuremath{-}}
\newunicodechar{⟨}{\ensuremath{\langle}}
\newunicodechar{⟩}{\ensuremath{\rangle}}
\newunicodechar{⋃}{\ensuremath{\bigcup}}
\newunicodechar{⋯}{\ensuremath{\cdots}}
\newunicodechar{∥}{\ensuremath{\parallel}}
\newunicodechar{̂}{\ensuremath{\hat{}}}
\newunicodechar{̃}{\ensuremath{\widetilde{}}}
\title{\bfseries Frattini Geometry of Non-Generating Complexes:\\Symmetry, Homology and Betti Numbers}
\author{Luis Alfredo Dupont García\thanks{Corresponding author: \href{mailto:ldupont@uv.mx}{ldupont@uv.mx}}\\[0.35em]
\small Facultad de Matemáticas, Universidad Veracruzana\\
\small Xalapa, Veracruz, Mexico}
\date{\small Preprint --- October 2026}
\begin{document}
\maketitle
\begin{center}
\small ORCID: \href{https://orcid.org/0000-0002-9044-3547}{0000-0002-9044-3547}
\end{center}
\vspace{-0.5em}
\begin{abstract}
For a finite group \(G\), let \(N(G)\) be the simplicial complex of
subsets that do not generate \(G\). We separate a general
Frattini-topological reduction from the structure special to finite
\(p\)-groups. For every finite noncyclic group, the non-cone Frattini
core is homotopy equivalent to the order complex of the proper part of
the subgroup lattice of \(G/\Phi(G)\). For a finite \(p\)-group this
quotient is an elementary abelian vector space, and the core becomes
the non-spanning complex of a uniform parallel extension of
\(\mathrm{PG}(r-1,p)\), where \(r=d(G)\) and \(q=|\Phi(G)|\).
We use this geometry to determine the exact simplicial isomorphism data
and full simplicial automorphism group, identify top homology
equivariantly with the appropriate restriction of the Steinberg module,
and derive modular consequences. We also give an explicit supportwise
formula for every multigraded Betti number and compare it with a closed
single-sum formula for the \(\mathbb Z\)-graded Betti numbers. Known
matroidal and building-theoretic inputs are stated separately from the
group-specific consequences. The resulting framework also yields the
homotopy type, depth, regularity, projective dimension, face
enumeration, and the complete Cohen--Macaulay classification.
\end{abstract}

\section{Introduction}\label{introduction}

Simplicial and hypergraph structures attached to generating sets of
finite groups have recently received systematic attention \cite{ref1,ref2,ref3,ref4}.
For finite \(p\)-groups, the generating structure is matroidal \cite{ref2,ref3},
while the topology of non-spanning matroid complexes, linear
resolutions of matroidal ideals, and Tutte-theoretic descriptions of
\(\Tor\) provide established combinatorial and algebraic input
\cite{ref5,ref6,ref7,ref8,ref9,ref19,ref22}. The purpose of this paper is not to reprove those
theories, but to isolate the Frattini geometry that links them and to
extract group-theoretic, equivariant, and multigraded consequences.

The first point is independent of the prime-power hypothesis. If \(G\)
is finite and noncyclic, \(\Phi(G)\) consists precisely of the cone
vertices of \(N(G)\), and after deleting them the resulting Frattini
core is controlled by the interval \([\Phi(G),G]\) in the subgroup
lattice. For \(p\)-groups, Burnside's basis theorem turns this
quotient-lattice picture into projective geometry over
\(\mathbb F_p\). This is the step that makes the later formulas
explicit.

\medskip\noindent\textbf{Known input and new contributions.}
We use as known input the matroidal nature of generating sets of finite
\(p\)-groups \cite{ref2,ref3}, standard topology of non-spanning complexes
\cite{ref9,ref26}, the Solomon--Tits theorem \cite{ref10,ref14}, linear resolutions of
matroidal ideals \cite{ref7,ref8}, Hochster's formula \cite{ref6}, and the NBC/Tutte
description of \(\Tor\) \cite{ref19}. A useful structural device is the
Frattini-core reduction of Section~3, obtained from the nerve and
crosscut theorems and placed in the context of subgroup-poset topology
\cite{ref16,ref17,ref21}. The strongest algebraic refinement is supportwise: for
every squarefree multidegree \(W\) we determine \(\beta_{i,W}\)
through the Möbius invariant of the projective support configuration.
We then derive exact \(p\)-group reconstruction and non-reconstruction
statements, the full simplicial symmetry description, an equivariant
Steinberg formulation with realization and modular consequences, and
closed graded formulas in the parameters \((p,r,q)\).

The paper is organized so that these two layers remain visible.
Section~3 contains the finite-group Frattini reduction. Sections~4
and~5 exploit the elementary-abelian Frattini quotient of a
\(p\)-group to obtain projective rigidity, simplicial symmetries and
equivariant top homology. Sections~6 and~7 pass to Stanley--Reisner
theory, first at the multigraded level and then to closed
\(\mathbb Z\)-graded formulas. Section~8 records the depth and
Cohen--Macaulay consequences, and Section~9 gives numerical checks.

\section{Background and standard matroidal input}
\label{background-and-standard-matroidal-input}

\subsection{Non-generating complexes and effective vertices}
\label{non-generating-complexes-and-effective-vertices}

For a finite group \(G\), define
\[
N(G)=\{A\subseteq G:\langle A\rangle\neq G\},
\]
using the effective vertex set
\[
V(N(G))=\{g\in G:\langle g\rangle\neq G\}.
\]
The Stanley--Reisner ring \(K[N(G)]\) is always taken over the
polynomial ring on these effective vertices. If \(G\) is noncyclic,
every element is an effective vertex and the facets are exactly the
maximal subgroups.

\subsection{Standard facts for non-spanning matroid complexes}
\label{standard-facts-for-non-spanning-matroid-complexes}

Let \(M\) be a loopless rank-\(r\) matroid on a ground set \(E\) of
cardinality \(n\), and put \(\Delta=\NS(M)\). We use the following
standard input.

\begin{enumerate}[label=(\alph*)]
\item By the crosscut description of a geometric lattice and matroid
shellability, \(\Delta\) is homotopy equivalent to the order complex
of the proper part of \(L_M\), and hence to a wedge of
\[
\left|\mu_{L_M}(\hat0,\hat1)\right|
\]
spheres of dimension \(r-2\) \cite{ref9,ref26}.

\item The \((r-2)\)-skeleton of \(\Delta\) is the complete
\((r-2)\)-skeleton on \(E\), hence is Cohen--Macaulay. Hibi's
skeleton-depth criterion \cite{ref27} therefore gives
\(\depth K[\Delta]\ge r-1\). On the other hand, Hochster's formula
\cite{ref6} and the nonzero top reduced homology give
\[
\beta_{n-r+1,E}\neq0,
\]
so \(\pd K[\Delta]\ge n-r+1\). Auslander--Buchsbaum then forces
\[
\pd K[\Delta]=n-r+1,
\qquad
\depth K[\Delta]=r-1.
\]
The terminal Betti number is
\(\left|\mu_{L_M}(\hat0,\hat1)\right|\) \cite{ref22}.

\item The ring \(K[\Delta]\) is Cohen--Macaulay exactly when \(M\) is
uniform. Indeed, \(\dim K[\Delta]\) is the maximum cardinality of a
hyperplane, whereas \(\depth K[\Delta]=r-1\). Equality forces every
hyperplane to have \(r-1\) elements. If \(M\) were not uniform, it
would have a circuit \(C\) with \(|C|\le r\). Choose \(e\in C\) and
extend the independent set \(C\setminus\{e\}\) to a hyperplane \(H\).
Because \(e\in\operatorname{cl}(C\setminus\{e\})\), the hyperplane
contains all of \(C\). If \(|C|<r\), extending the rank
\(|C|-1\) flat \(\operatorname{cl}(C)\) to rank \(r-1\) requires
\(r-|C|\) additional independent elements outside
\(\operatorname{cl}(C)\). Thus
\[
|H|\ge |C|+(r-|C|)=r,
\]
a contradiction. Hence \(M=U_{r,n}\); the converse is
immediate.
\end{enumerate}

These facts are recalled only to make the specialization boundary
explicit.

\section{Frattini cores of arbitrary finite groups}
\label{frattini-cores-of-arbitrary-finite-groups}

Let \(G\) be a finite noncyclic group, put \(\Phi=\Phi(G)\), and define
\[
C_{\Phi}(G):=N(G)|_{G\setminus\Phi}.
\]

\begin{proposition}[join decomposition]\label{prop:join}
One has
\[
N(G)=2^{\Phi}*C_{\Phi}(G).
\]
Hence \(K[N(G)]\) is a polynomial extension of \(K[C_{\Phi}(G)]\).
\end{proposition}

\begin{proof}
Every maximal subgroup contains \(\Phi\), so every element of \(\Phi\)
is a cone vertex. Conversely, a face splits uniquely into its
\(\Phi\)-part and its part outside \(\Phi\).
\end{proof}

\begin{proposition}[structural Frattini-core reduction]
\label{prop:frattini-core}
Let \(G\) be a finite noncyclic group. There is a natural homotopy equivalence
\begin{equation}\label{eq:frattini-core}
C_{\Phi}(G)\simeq
\Delta((\Phi(G),G))
\cong
\Delta(\overline L(G/\Phi(G))).
\end{equation}
\end{proposition}

\begin{proof}
Every proper subgroup of a finite group is contained in a maximal
subgroup. Hence \(A\subseteq G\setminus\Phi\) fails to generate \(G\)
exactly when it is contained in \(M\setminus\Phi\) for some maximal
subgroup \(M<G\). Therefore
\[
C_{\Phi}(G)
=
\bigcup_{M\ \mathrm{max}\ G}\Delta_M,
\qquad
\Delta_M=2^{\,M\setminus\Phi}.
\]
For maximal subgroups \(M_1,\ldots,M_s\),
\[
\Delta_{M_1}\cap\cdots\cap\Delta_{M_s}
=
2^{(M_1\cap\cdots\cap M_s)\setminus\Phi}.
\]
As an abstract simplicial complex this intersection always contains
the empty face; the condition relevant to the nerve lemma is that it
contain a vertex. This occurs exactly when
\[
M_1\cap\cdots\cap M_s>\Phi.
\]
Whenever that happens, the intersection is a nonempty simplex and is
therefore contractible.

Consequently the nerve of the cover has one vertex for each maximal
subgroup of \(G\), and a face \(\{M_1,\ldots,M_s\}\) precisely when
\(M_1\cap\cdots\cap M_s>\Phi\). This is the coatom crosscut complex
of the interval \([\Phi,G]\). The crosscut theorem \cite{ref9} identifies its
homotopy type with the order complex \(\Delta((\Phi,G))\).

Finally, the correspondence theorem gives a lattice isomorphism
\[
[\Phi,G]\longrightarrow L(G/\Phi),
\qquad
H\longmapsto H/\Phi.
\]
Removing the bottom and top elements proves
\eqref{eq:frattini-core}. The construction is natural with respect to
automorphisms of \(G\), a point used later in the equivariant section.
For related subgroup-lattice topology and shellability, see
\cite{ref16,ref17,ref21}.
\end{proof}

\begin{corollary}[Euler characteristic of the Frattini core]
\label{cor:euler-frattini}
For every finite noncyclic group \(G\),
\[
\widetilde\chi(C_{\Phi}(G))
=
\mu_{L(G)}(\Phi(G),G)
=
\mu_{L(G/\Phi(G))}(1,G/\Phi(G)).
\]
Thus the abstract complex \(N(G)\) determines the Hall Möbius value
of the Frattini quotient.
\end{corollary}

\begin{proof}
The reduced Euler characteristic of the order complex of the open
interval \((\Phi(G),G)\) is the Möbius value
\(\mu_{L(G)}(\Phi(G),G)\). Proposition~\ref{prop:frattini-core} and
the lattice isomorphism
\([\Phi(G),G]\cong L(G/\Phi(G))\) give the equalities. This connects
the Frattini-core reduction with Hall's Eulerian/Möbius framework \cite{ref20}
and Brown's probabilistic-zeta viewpoint \cite{ref21}.
\end{proof}

\begin{proposition}[general reconstruction data]
\label{prop:general-reconstruction}
If \(G\) is finite and noncyclic, the abstract complex \(N(G)\)
determines \(|G|\) and \(|\Phi(G)|\). Indeed, \(|G|\) is the number
of vertices, while \(\Phi(G)\) is the set of vertices lying in every
facet.
\end{proposition}

\section{Finite \(p\)-groups: projective structure and reconstruction}
\label{finite-$p$-groups-projective-structure-and-reconstruction}

Assume now that \(G\) is a finite \(p\)-group. Put
\[
\Phi=\Phi(G),\qquad r=d(G),\qquad q=|\Phi|,\qquad
V=G/\Phi\cong\mathbb F_p^r,\qquad E=G\setminus\Phi,
\]
and set \(C(G):=C_\Phi(G)\). From now on, except in Theorem~4.2 and Theorem~8.1, we assume that \(G\) is noncyclic, i.e. \(r\ge2\). Burnside's basis theorem says that
\(A\subseteq G\) generates \(G\) exactly when its image spans \(V\).
Hence, for \(A\subseteq E\),
\[
\operatorname{rk}_G(A)
=
\dim_{\mathbb F_p}\operatorname{span}\{a\Phi:a\in A\}
\]
defines a rank-\(r\) linear matroid \(M_G\) \cite{ref2,ref3,ref5}, and
\(C(G)=\NS(M_G)\).

For each one-dimensional subspace \(L\le V\), the fiber
\[
E_L=\{g\in E:g\Phi\in L\setminus\{0\}\}
\]
has cardinality \(q(p-1)\). These fibers are precisely the parallel
classes of \(M_G\); after simplification one obtains
\(\mathrm{PG}(r-1,p)\). Thus \(M_G\) is the uniform parallel extension
of \(\mathrm{PG}(r-1,p)\) with multiplicity
\[
m=q(p-1).
\]

This is the point at which the \(p\)-group hypothesis becomes
decisive. For an arbitrary finite group, \(G/\Phi(G)\) need not be
elementary abelian, so the projective-matroid model used below is no
longer available in this form; compare \cite{ref3}. In the \(p\)-group case,
however, the entire core is determined by \((p,r,q)\).

\medskip\noindent\textbf{Corollary 4.1 (explicit homotopy type).}
If \(r\ge3\), then
\[
C(G)\simeq
\bigvee^{\,p^{\binom r2}} S^{r-2}.
\]
If \(r=2\), then \(C(G)\) is the disjoint union of \(p+1\) simplices,
and therefore
\[
\widetilde H_0(C(G);K)\cong K^p.
\]

\begin{proof}
By Proposition~3.2, \(C(G)\) is homotopy equivalent to the order
complex of the proper subspace lattice of \(V\). For \(r\ge3\), this
is the spherical building of \(\mathrm{GL}(r,p)\). The Solomon--Tits
theorem gives a wedge of \((r-2)\)-spheres, and
\[
\mu(0,V)=(-1)^r p^{\binom r2},
\]
so the number of spheres is \(p^{\binom r2}\). For \(r=2\), the
proper nonzero subspaces are the \(p+1\) projective points, and the
corresponding core is a disjoint union of \(p+1\) contractible
simplices.
\end{proof}

\medskip\noindent\textbf{Theorem 4.2 (complete classification for nontrivial finite \(p\)-groups).}
For \(i=1,2\), let \(G_i\) be a nontrivial finite \(p_i\)-group, put
\[
r_i=d(G_i),\qquad q_i=|\Phi(G_i)|.
\]
Then \(N(G_1)\cong N(G_2)\) if and only if either both groups are
cyclic and \(q_1=q_2\), or both are noncyclic and
\[
(p_1,r_1,q_1)=(p_2,r_2,q_2).
\]

\begin{proof}
If \(G\) is cyclic of order \(p^a\), its effective vertex set is
\(\Phi(G)\), and every subset of that set is a face. Hence \(N(G)\)
is a simplex on \(q=|\Phi(G)|\) vertices.

Now suppose that \(G\) is noncyclic. Then every group element is an
effective vertex, so the number of vertices gives
\[
|G|=qp^r.
\]
The cone vertices are exactly \(\Phi(G)\), hence the abstract complex
determines \(q\). Therefore it also determines
\[
p^r=\frac{|G|}{q}>1.
\]
Since this integer is a prime power, its unique prime divisor is
\(p\), and the corresponding exponent is \(r\). Conversely,
the construction above shows that \((p,r,q)\) determines the core and
the \(q\)-vertex simplex factor. Finally, a noncyclic \(p\)-group has
at least \(p+1\) maximal subgroups, so its complex cannot be a
simplex.
\end{proof}

\medskip\noindent\textbf{Corollary 4.3 (non-reconstruction).}
Distinct groups may have the same complex. For \(p=2\),
\[
C_4\times C_2,\qquad D_8,\qquad Q_8
\]
share \((p,r,q)=(2,2,2)\). For odd \(p\),
\[
C_{p^2}\times C_p,\qquad UT_3(p),\qquad M_{p^3}
\]
all have \(d=2\) and \(|\Phi|=p\), while the two nonabelian examples
have exponents \(p\) and \(p^2\), respectively. Thus \(N(G)\) does
not in general detect isomorphism type, abelianity, or exponent.

\section{Simplicial automorphisms and equivariant
homology}\label{simplicial-automorphisms-and-equivariant-homology}

Let
\[
\Omega=\mathrm{PG}(r-1,p),\qquad
\theta_r(p)=\frac{p^r-1}{p-1},\qquad
m=q(p-1).
\]
A simplicial automorphism of \(C(G)\) preserves facets, hence
hyperplanes, spanning sets and bases; therefore
\(\Aut(C(G))=\Aut(M_G)\). The parallel classes are intrinsic as
equality classes of facet-incidence patterns.

\begin{lemma}[recovery of the matroid from the non-spanning complex]
\label{lem:recover-matroid}
Let \(M\) be a loopless rank-\(r\) matroid. A permutation of the
ground set is a simplicial automorphism of \(\NS(M)\) if and only if
it is a matroid automorphism of \(M\). Consequently
\[
\Aut(\NS(M))=\Aut(M).
\]
\end{lemma}

\begin{proof} A subset A is spanning in M exactly when A is not a face of
NS(M). Hence every simplicial automorphism of NS(M) preserves spanning
subsets. The bases are precisely the inclusion-minimal spanning subsets,
so they are preserved as well. A permutation preserving the set of bases
is a matroid automorphism. The converse is immediate because a matroid
automorphism preserves rank and therefore the non-spanning subsets. \end{proof}

\medskip\noindent\textbf{Theorem 5.2 (simplicial automorphisms).}
One has
\[
\Aut(N(G))\cong S_q\times\Aut(C(G)).
\]
For \(r\ge3\),
\[
\Aut(C(G))
\cong S_m\wr_{\Omega}\mathrm{PGL}(r,p)
\cong (S_m)^{\theta_r(p)}\rtimes\mathrm{PGL}(r,p),
\]
whereas for \(r=2\),
\[
\Aut(C(G))\cong S_m\wr S_{p+1}.
\]
Under the natural action on the projective line,
\(\mathrm{PGL}(2,p)\le S_{p+1}\), with equality exactly for
\(p\in\{2,3\}\).

\begin{proof}
Cone vertices are intrinsic, so every automorphism of \(N(G)\)
preserves \(\Phi\), and arbitrary permutations of \(\Phi\) give the
direct factor \(S_q\). On \(E\), Lemma~5.1 identifies
\(\Aut(C(G))\) with \(\Aut(M_G)\).

The parallel classes of \(M_G\) are intrinsic. For distinct
\(x,y\in E\), one has \(x\parallel y\) exactly when \(x\) and \(y\)
lie in the same hyperplanes, equivalently in the same facets of
\(C(G)\). Indeed, if \(y\notin\operatorname{cl}\{x\}\), a hyperplane
containing \(\operatorname{cl}\{x\}\) but not \(y\) separates their
facet-incidence patterns. Thus a matroid automorphism permutes the
projective fibers and may independently permute the \(m\) elements
inside each fiber.

Collapsing the fibers gives the simplification
\(\mathrm{PG}(r-1,p)\). For \(r\ge3\), an automorphism preserves
rank-two flats and hence projective lines; by the fundamental theorem
of projective geometry \cite{ref24} it is induced by a semilinear
transformation. Since \(\mathbb F_p\) has no nontrivial field
automorphisms, the induced group is \(\mathrm{PGL}(r,p)\). For \(r=2\)
the simplification is \(U_{2,p+1}\), whose full automorphism group is
\(S_{p+1}\). Choosing identifications of the equal fibers gives the
displayed wreath-product descriptions.
\end{proof}

\subsection{Equivariant Steinberg homology}
\label{equivariant-steinberg-homology}

Let
\[
\rho_G:\Aut(G)\longrightarrow\mathrm{GL}(V)
\]
be the induced action and put
\[
\St_{V,K}:=
\widetilde H_{r-2}\!\left(\Delta(\overline L(V));K\right).
\]
The hyperplane cover and the crosscut construction are natural under
automorphisms.

\medskip\noindent\textbf{Theorem 5.3 (equivariant Frattini-core theorem).}
For every field \(K\) and \(r\ge2\),
\begin{equation}\label{eq:steinberg}
\widetilde H_{r-2}(C(G);K)
\cong
\rho_G^*\!\left(
\operatorname{Res}^{\mathrm{GL}(V)}_{\rho_G(\Aut(G))}
\St_{V,K}
\right)
\end{equation}
as \(K[\Aut(G)]\)-modules. In particular,
\(\ker\rho_G\) acts trivially and
\[
\dim_K\widetilde H_{r-2}(C(G);K)=p^{\binom r2}.
\]
If \(G=C_p^r\), this is the full Steinberg module.

\begin{proof}
Every \(\alpha\in\Aut(G)\) preserves \(\Phi(G)\), so it induces
\(\rho_G(\alpha)\in\mathrm{GL}(V)\). For each hyperplane \(H<V\),
the automorphism \(\alpha\) sends the corresponding simplex of the
hyperplane cover to that indexed by \(\rho_G(\alpha)H\). Hence the
cover is \(\Aut(G)\)-stable, and the induced action on its nerve is
the action obtained from \(\rho_G\) on the hyperplanes of \(V\).

The nerve and crosscut equivalences can be chosen equivariantly; see
\cite{ref12}. The resulting representation-theoretic viewpoint is also
consistent with the matroid-homology framework of Moci--Pezzoli \cite{ref18}. Therefore they induce isomorphisms of
\(K[\Aut(G)]\)-modules on reduced homology. The proper subspace lattice
of \(V\) is the spherical building for \(\mathrm{GL}(V)\), and its
top reduced homology is the Steinberg module. Since the action of
\(\Aut(G)\) on the building factors through \(\rho_G\), the resulting
module is precisely the inflation displayed in
\eqref{eq:steinberg}. The dimension formula follows from
Solomon--Tits.
\end{proof}

\medskip\noindent\textbf{Proposition 5.4 (full simplicial action).}
For \(r\ge3\), the normal subgroup
\((S_m)^{\theta_r(p)}\le\Aut(C(G))\) acts trivially on top reduced
homology, so the full simplicial action factors through
\(\mathrm{PGL}(r,p)\). For \(r=2\),
\[
\widetilde H_0(C(G);K)
\cong
\ker\!\left(K^{p+1}\xrightarrow{\sum}K\right)
\]
with the natural \(S_{p+1}\)-action.

\begin{proof}
Internal permutations of a projective fiber fix every member of the
hyperplane cover setwise and hence act trivially on its nerve. The
rank-two statement follows from the \(p+1\) connected components.
\end{proof}

\subsection{Realization and modular consequences}
\label{realization-and-modular-consequences}

Bryant and Kov\'acs proved that for each linear group \(H\) of finite
dimension \(r\ge2\) over \(\mathbb F_p\), there exists a finite
\(p\)-group \(Q\) for which the group induced by \(\Aut(Q)\) on
\(Q/\Phi(Q)\) is exactly \(H\) \cite[Theorem~1]{ref13}. This concerns the
image of the induced action; it does not assert \(\Aut(Q)\cong H\).

\medskip\noindent\textbf{Corollary 5.5 (realization of Steinberg restrictions).}
Let \(H\le\mathrm{GL}(r,p)\), \(r\ge2\), and let \(K\) be any field.
There exists a finite \(p\)-group \(Q\) such that, after inflation
along \(\Aut(Q)\to H\), the module
\[
\operatorname{Res}^{\mathrm{GL}(r,p)}_H\St_{r,K}
\]
occurs as \(\widetilde H_{r-2}(C(Q);K)\).

\begin{proof}
Choose \(Q\) by the Bryant--Kov\'acs realization theorem and apply
Theorem~5.3.
\end{proof}

\medskip\noindent\textbf{Corollary 5.6 (\(p\)-local freeness).}
Let \(K\) be a field of characteristic \(p\) and let
\[
U\le\rho_G(\Aut(G))\le\mathrm{GL}(V)
\]
be a \(p\)-subgroup. Since \(\ker\rho_G\) acts trivially by
Theorem~5.3, we regard \(\widetilde H_{r-2}(C(G);K)\) as a
\(K[\rho_G(\Aut(G))]\)-module. Its restriction to \(U\) is a free
\(K[U]\)-module of rank
\[
\frac{p^{\binom r2}}{|U|}.
\]

\begin{proof}
The Steinberg module in defining characteristic is projective over
\(K[\mathrm{GL}(r,p)]\) \cite[Chapter~9]{ref25}. Restriction to \(U\)
preserves projectivity because \(K[\mathrm{GL}(r,p)]\) is free as a
right \(K[U]\)-module. Since \(K[U]\) is local for a finite
\(p\)-group in characteristic \(p\), every finitely generated
projective \(K[U]\)-module is free. Comparing dimensions gives the
stated rank. If \(U\) is a Sylow \(p\)-subgroup of
\(\mathrm{GL}(r,p)\), then \(|U|=p^{\binom r2}\), so the restriction
is the regular \(K[U]\)-module.
\end{proof}

\section{Matroidal resolutions and multigraded Betti numbers}
\label{matroidal-resolutions-and-multigraded-betti-numbers}

Let
\[
S=K[x_g:g\in E],
\qquad
R=K[C(G)].
\]
The minimal nonfaces of \(C(G)=\NS(M_G)\) are exactly the bases of
\(M_G\), so
\[
I_{C(G)}=(x_B:B\text{ a basis of }M_G).
\]
This is the basis ideal of a matroid. By the basis-exchange axiom it
is a matroidal ideal; matroidal ideals have linear quotients, and
because all generators have degree \(r\), the ideal \(I_{C(G)}\) has
an \(r\)-linear minimal free resolution \cite{ref7,ref8}. Consequently
\[
\beta_{i,j}(R)=0
\quad\text{for }i\ge1\text{ unless }j-i=r-1,
\qquad
\reg R=r-1.
\]

For \(W\subseteq E\), let \(P(W)\) be the set of projective points met
by \(W\), and let
\[
M(W):=\operatorname{si}(M_G|W)
\]
be the simplification of the restriction.

\begin{theorem}[explicit multigraded Betti numbers]\label{thm:multigraded-betti}
For \(i\ge1\),
\begin{equation}\label{eq:multigraded-betti}
\beta_{i,W}(R)=
\begin{cases}
\left|\mu_{L(M(W))}(\hat0,\hat1)\right|,
& \operatorname{rk}M(W)=r,\quad |W|=i+r-1,\\[2mm]
0,&\text{otherwise}.
\end{cases}
\end{equation}

\begin{proof}
Hochster's formula in squarefree multidegree \(W\) gives
\[
\beta_{i,W}(R)
=
\dim_K
\widetilde H_{|W|-i-1}(C(G)[W];K).
\]
If \(\operatorname{rk}_{M_G}(W)<r\), every subset of \(W\) is
non-spanning, so \(C(G)[W]=2^W\) is a simplex and the relevant reduced
homology vanishes.

Suppose \(\operatorname{rk}_{M_G}(W)=r\). Restriction commutes with
the non-spanning condition, hence
\[
C(G)[W]=\NS(M_G|W).
\]
The restriction is loopless of rank \(r\), so its non-spanning
complex has reduced homology only in degree \(r-2\), where
\[
\dim_K\widetilde H_{r-2}(C(G)[W];K)
=
\left|\mu_{L(M_G|W)}(\hat0,\hat1)\right|.
\]
Thus Hochster's term can be nonzero only when
\(|W|-i-1=r-2\), equivalently \(|W|=i+r-1\). Passing to the
simplification \(M(W)\) only identifies parallel elements and leaves
the lattice of flats unchanged up to canonical isomorphism. This
proves \eqref{eq:multigraded-betti}.
\end{proof}
\end{theorem}

\begin{remark}[Tutte evaluation]
In the nonzero case of Theorem~\ref{thm:multigraded-betti} one may equivalently write
\[
\beta_{i,W}(R)=T_{M(W)}(1,0),
\]
because for a loopless rank-\(r\) matroid \(M\),
\[
\mu_{L(M)}(\hat0,\hat1)=(-1)^rT_M(1,0).
\]
This Tutte evaluation should not be confused with Crapo's beta
invariant.
\end{remark}

\begin{corollary}[support symmetry]\label{cor:support-symmetry}
If \(W,W'\subseteq E\) satisfy
\[
|W|=|W'|,
\qquad
P(W)=P(W'),
\]
then \(\beta_{i,W}(R)=\beta_{i,W'}(R)\) for every \(i\).
Moreover, every \(\sigma\in\Aut(C(G))\), acting by permutation of the
variables, induces isomorphisms
\[
\Tor_i^S(R,K)_W
\cong
\Tor_i^S(R,K)_{\sigma(W)}.
\]
Consequently \(\beta_{i,W}=\beta_{i,\sigma(W)}\), and the stabilizer
of \(W\) acts on
\[
\Tor_i^S(R,K)_W
\cong
\widetilde H_{|W|-i-1}(C(G)[W];K).
\]
\end{corollary}

\section{Explicit face and graded Betti
formulas}\label{explicit-face-and-graded-betti-formulas}

\medskip\noindent\textbf{Proposition 7.1 (face enumeration).}
Put \(n=q(p^r-1)\), and let \(S_{p,r,q}(k)\) be the number of
\(k\)-subsets of \(E\) whose images span \(V\). Then
\begin{equation}\label{eq:spanning-count}
S_{p,r,q}(k)=
\sum_{j=0}^{r}
\qbinom{r}{j}
(-1)^{r-j}p^{\binom{r-j}{2}}
\binom{q(p^{j}-1)}{k}.
\end{equation}
Consequently,
\[
f_{k-1}(C(G))=\binom{n}{k}-S_{p,r,q}(k).
\]

\begin{proof}
For each \(j\)-dimensional subspace \(U\le V\), its nonzero vectors
have \(q(p^j-1)\) lifts in \(E\). Möbius inversion on the subspace
lattice, using
\[
\mu(U,V)=(-1)^c p^{\binom{c}{2}},
\qquad c=\operatorname{codim}(U),
\]
counts the \(k\)-subsets whose image spans \(V\) and gives
\eqref{eq:spanning-count}; see \cite{ref23}. Subtracting these
spanning subsets from all \(\binom{n}{k}\) \(k\)-subsets gives the
\(f\)-vector. Moreover, \(k!S_{p,r,q}(k)\) counts ordered generating
\(k\)-tuples with pairwise distinct entries from \(E\), clarifying the
relation with Hall's Eulerian function for unrestricted ordered tuples.
\end{proof}

\medskip\noindent\textbf{Theorem 7.2 (closed graded Betti formula).}
For \(1\le i\le n-r+1\),
\begin{equation}\label{eq:graded-betti}
\beta_{i,i+r-1}(R)=
\sum_{j=0}^{r-1}
(-1)^j\qbinom{r}{j}
p^{\binom{r-j}{2}}
\binom{q(p^r-p^j)}{i+r-1}.
\end{equation}

\begin{proof}
Since the resolution is linear, write
\(\beta_i=\beta_{i,i+r-1}(R)\). The Hilbert numerator satisfies
\[
1+\sum_{i\ge1}(-1)^i\beta_i t^{i+r-1}
=
\sum_{k\ge0}f_{k-1}t^k(1-t)^{n-k}.
\]
Substitute \(f_{k-1}=\binom nk-S_{p,r,q}(k)\).
For positive degrees the contribution of \(\binom nk\) cancels.
Interchanging the finite sums and using
\[
\sum_k(-1)^k
\binom Nk\binom{n-k}{m-k}
=
\binom{n-N}{m}
\]
reduces the contribution of a \(j\)-dimensional subspace to
\(\binom{q(p^r-p^j)}m\).

Now set \(m=i+r-1\) and
\[
c_j=\qbinom rj(-1)^{r-j}p^{\binom{r-j}{2}}.
\]
Coefficient comparison gives
\[
-(-1)^m\sum_j c_j\binom{q(p^r-p^j)}m
=
(-1)^i\beta_i.
\]
Since \(m=i+r-1\), this is equivalent to
\[
\beta_i=(-1)^r\sum_jc_j\binom{q(p^r-p^j)}m,
\]
which is exactly \eqref{eq:graded-betti}.
\end{proof}

\medskip\noindent\textbf{Corollary 7.3 (configuration identity).}
For \(1\le i\le n-r+1\),
\begin{equation}\label{eq:configuration-identity}
\sum_{\substack{W\subseteq E\\ |W|=i+r-1\\ \operatorname{rk}_G(W)=r}}
\left|\mu_{L(M(W))}(\hat0,\hat1)\right|
=
\sum_{j=0}^{r-1}
(-1)^j\qbinom rj
p^{\binom{r-j}{2}}
\binom{q(p^r-p^j)}{i+r-1}.
\end{equation}
In particular, the alternating Gaussian expression is a nonnegative
integer. It is in fact strictly positive throughout the stated range:
every cardinality between \(r\) and \(n\) occurs for a support
containing a basis of \(M_G\), and the Möbius invariant of the
resulting geometric lattice is nonzero by Rota's sign theorem \cite{ref11}.

\begin{proof}
By multigraded decomposition,
\[
\beta_{i,i+r-1}(R)
=
\sum_{\substack{W\subseteq E\\ |W|=i+r-1}}
\beta_{i,W}(R).
\]
Theorem~\ref{thm:multigraded-betti} identifies the nonzero summands with the absolute Möbius
invariants on the left of \eqref{eq:configuration-identity}, whereas
Theorem~7.2 evaluates the same graded Betti number by the right-hand
side.
\end{proof}

\medskip\noindent\textbf{Corollary 7.4 (extremal terms).}
One has
\[
\beta_{1,r}(R)=\frac{q^r|\mathrm{GL}(r,p)|}{r!},
\qquad
\beta_{n-r+1,n}(R)=p^{\binom r2}.
\]

\begin{proof}
The number \(\beta_{1,r}\) counts bases of \(M_G\): choose an ordered
basis of \(V\), lift each vector in \(q\) ways, and divide by \(r!\).
For \(i=n-r+1\), formula \eqref{eq:graded-betti} has \(m=n\); only the
\(j=0\) term survives.
\end{proof}

\subsection{Compatibility with the NBC/Tutte
description}\label{compatibility-with-the-nbctutte-description}

Binder's NBC/Tutte formula \cite[Chapter~5]{ref19} computes
the same graded Tor groups over Q. We use it only as a compatibility
check: after translating to the grading convention adopted here, it
gives the same Betti numbers as Theorem 7.2. The contribution of the
present formula is its explicit evaluation for the uniform parallel
extension of projective geometry.

\section{Depth, Cohen--Macaulayness and extremal invariants}
\label{depth-cohen-macaulayness-and-extremal-invariants}

By Section~2.2(b) and Theorem~\ref{thm:multigraded-betti}, the standard
non-spanning-complex results specialize to
\[
\depth R=r-1,\qquad
\reg R=r-1,\qquad
\operatorname{pd}_S R=n-r+1.
\]
For the full complex, the join with the \(q\)-vertex simplex
\(2^\Phi\) is a polynomial extension, so
\[
\depth K[N(G)]=q+r-1,\qquad
\reg K[N(G)]=r-1,
\]
while the projective dimension over the polynomial ring on
\(|G|=n+q\) variables remains \(n-r+1\).

\medskip\noindent\textbf{Theorem 8.1 (Cohen--Macaulay classification).}
For every nontrivial finite \(p\)-group \(G\) and every field \(K\),
the ring \(K[N(G)]\) is Cohen--Macaulay if and only if \(G\) is
cyclic or
\[
G\cong C_2\times C_2.
\]

\begin{proof}
For cyclic \(G\), the effective complex is a simplex. Assume \(r\ge2\).
Polynomial extension by the cone variables preserves and reflects
Cohen--Macaulayness, so it suffices to consider \(C(G)\).
The largest hyperplane of \(M_G\) has
\[
q(p^{r-1}-1)
\]
elements, hence
\[
\dim K[C(G)]=q(p^{r-1}-1),
\qquad
\depth K[C(G)]=r-1.
\]
Thus Cohen--Macaulayness is equivalent to
\[
r-1=q(p^{r-1}-1).
\]
If \(r\ge3\), then
\[
p^{r-1}-1\ge2^{r-1}-1\ge r>r-1,
\]
so equality is impossible. Hence \(r=2\), and the condition becomes
\(1=q(p-1)\), forcing \(p=2\) and \(q=1\). Therefore
\(G\cong C_2\times C_2\).
\end{proof}

\section{Examples}
\label{examples}

\medskip\noindent\textbf{Example 9.1.}
For \(C_4\times C_2\), \(D_8\), and \(Q_8\),
\[
(p,r,q)=(2,2,2),\qquad n=6,\qquad f=(1,6,3),
\]
and the positive graded Betti numbers are
\[
\beta_{1,2}=12,\quad
\beta_{2,3}=28,\quad
\beta_{3,4}=27,\quad
\beta_{4,5}=12,\quad
\beta_{5,6}=2.
\]

\medskip\noindent\textbf{Example 9.2.}
For \(C_2^3\),
\[
(p,r,q)=(2,3,1),\qquad n=7,\qquad f=(1,7,21,7),
\]
and the positive Betti numbers along \(j-i=2\) are
\[
28,\ 77,\ 84,\ 42,\ 8.
\]

\medskip\noindent\textbf{Example 9.3.}
For \(C_3\times C_3\),
\[
(p,r,q)=(3,2,1),\qquad n=8,\qquad f=(1,8,4),
\]
and the positive Betti numbers along \(j-i=1\) are
\[
24,\ 88,\ 150,\ 144,\ 80,\ 24,\ 3.
\]

\medskip\noindent\textbf{Example 9.4 (\(q>p\) check).}
For
\[
(p,r,q)=(2,2,3),\qquad n=9,\qquad f=(1,9,9,3),
\]
the positive Betti sequence is
\[
(27,108,207,234,165,72,18,2),
\]
providing an additional check in a regime not covered by the preceding
examples.

\section{Concluding remarks and open
problems}\label{concluding-remarks-and-open-problems}

The Frattini-core viewpoint separates the topology available for
arbitrary finite noncyclic groups from the projective-matroidal
structure special to $p$-groups. For $p$-groups, the same geometry
simultaneously controls simplicial symmetry, Steinberg homology,
multigraded Tor and the closed Betti formulas. Natural next questions
include determining how much additional information is recovered by the
\(\operatorname{Aut}(G)\)-action when N(G) alone does not determine G; computing \(C_{\Phi}(G)\)
explicitly for finite nilpotent groups, where \(G/\Phi(G)\) is a product of
elementary abelian Sylow factors and the subgroup lattice can be
analyzed through products/joins of subspace lattices, in the spirit of
Quillen's poset-topological methods \cite{ref15}; and comparing the
non-generating complex with independence-type complexes that may retain
more of the multiplication law.

\section*{Author information}
Luis Alfredo Dupont García, Facultad de Matemáticas, Universidad Veracruzana,
Paseo 112, Lote 2, Col. Nuevo Xalapa, C.P. 91097, Xalapa, Veracruz, Mexico.\\
E-mail: \href{mailto:ldupont@uv.mx}{ldupont@uv.mx}.\\
ORCID: \href{https://orcid.org/0000-0002-9044-3547}{0000-0002-9044-3547}.

\section*{Funding}
No specific funding was received for this research.

\section*{Competing interests}
The author declares no competing financial or non-financial interests.

\section*{Data availability}
No datasets were generated or analyzed in this study.

\section*{Artificial intelligence disclosure}
Generative AI and AI-assisted technologies. During the preparation of this work, the author used OpenAI ChatGPT (GPT-5.6 Sol) to assist with language editing, academic writing refinement, and support in reviewing the clarity and consistency of the mathematical exposition. The author produced, reviewed, and verified the mathematical content, proofs, citations, originality, and final form of the manuscript, and assumes full responsibility for them.


\begin{thebibliography}{99}
\bibitem{ref19} K. Binder. \emph{Tor Groups of the Stanley--Reisner Ring of a Matroid}. Ph.D. thesis, The Ohio State University, Columbus, Ohio (2025), Chapter 5.
\bibitem{ref26} A. Björner. The homology and shellability of matroids and geometric lattices. In \emph{Matroid Applications}, Encyclopedia of Mathematics and its Applications 40, 226--283, Cambridge University Press, Cambridge (1992).
\bibitem{ref9} A. Björner. Topological methods. In \emph{Handbook of Combinatorics}, Vol. 2, 1819--1872, Elsevier, Amsterdam (1995).
\bibitem{ref10} K. S. Brown. \emph{Buildings}. Springer, New York (1989).
\bibitem{ref21} K. S. Brown. The coset poset and probabilistic zeta function of a finite group. \emph{Journal of Algebra}, 225 (2000), 989--1012.
\bibitem{ref13} R. M. Bryant; L. G. Kovács. Lie representations and groups of prime power order. \emph{Journal of the London Mathematical Society}, 17 (1978), no. 3, 415--421; \href{https://doi.org/10.1112/jlms/s2-17.3.415}{doi:10.1112/jlms/s2-17.3.415}.
\bibitem{ref1} P. J. Cameron. Simplicial complexes defined on groups. \emph{Open Mathematics}, 23 (2025), 20250171.
\bibitem{ref2} P. J. Cameron; A. Lakshmanan S.; M. V. Ajith. Hypergraphs defined on algebraic structures. \emph{Communications in Combinatorics and Optimization}, 11 (2026), 131--144.
\bibitem{ref20} P. Hall. The Eulerian functions of a group. \emph{The Quarterly Journal of Mathematics}, os-7 (1936), 134--151.
\bibitem{ref27} T. Hibi. Quotient algebras of Stanley--Reisner rings and local cohomology. \emph{Journal of Algebra}, 140 (1991), 336--343; \href{https://doi.org/10.1016/0021-8693(91)90159-6}{doi:10.1016/0021-8693(91)90159-6}.
\bibitem{ref7} J. Herzog; T. Hibi. \emph{Monomial Ideals}. Springer, London (2011).
\bibitem{ref8} J. Herzog; T. Hibi. Discrete polymatroids. \emph{Journal of Algebraic Combinatorics}, 16 (2002), 239--268.
\bibitem{ref24} J. W. P. Hirschfeld. \emph{Projective Geometries over Finite Fields}. Second edition, Oxford University Press, Oxford (1998).
\bibitem{ref25} J. E. Humphreys. \emph{Modular Representations of Finite Groups of Lie Type}. London Mathematical Society Lecture Note Series 326, Cambridge University Press, Cambridge (2006).
\bibitem{ref22} T. Johnsen; H. Verdure. Hamming weights and Betti numbers of Stanley--Reisner rings associated to matroids. \emph{Applicable Algebra in Engineering, Communication and Computing}, 24 (2013), 73--93.
\bibitem{ref16} C. Kratzer; J. Thévenaz. Type d'homotopie des treillis et treillis des sous-groupes d'un groupe fini. \emph{Commentarii Mathematici Helvetici}, 60 (1985), 85--106.
\bibitem{ref3} A. Lucchini. Generating hypergraphs of finite groups. \emph{Ars Combinatoria}, 161 (2024), 17--28.
\bibitem{ref4} A. Lucchini; M. Stanojkovski. Independence and strong independence complexes of finite groups. \emph{Journal of the London Mathematical Society}, 113 (2026), e70605.
\bibitem{ref6} E. Miller; B. Sturmfels. \emph{Combinatorial Commutative Algebra}. Springer, New York (2005).
\bibitem{ref18} L. Moci; G. M. Pezzoli. Representations of automorphism groups on the homology of matroids. \emph{European Journal of Combinatorics}, 94 (2021), 103312.
\bibitem{ref5} J. Oxley. \emph{Matroid Theory}. Second edition, Oxford University Press, Oxford (2011).
\bibitem{ref15} D. Quillen. Homotopy properties of the poset of nontrivial \(p\)-subgroups of a group. \emph{Advances in Mathematics}, 28 (1978), 101--128.
\bibitem{ref11} G. C. Rota. On the foundations of combinatorial theory I. Theory of Möbius functions. \emph{Zeitschrift für Wahrscheinlichkeitstheorie und Verwandte Gebiete}, 2 (1964), 340--368.
\bibitem{ref17} J. Shareshian. On the shellability of the order complex of the subgroup lattice of a finite group. \emph{Transactions of the American Mathematical Society}, 353 (2001), 2689--2703.
\bibitem{ref14} L. Solomon. The Steinberg character of a finite group with BN-pair. In \emph{Theory of Finite Groups}, 213--221, Benjamin, New York (1969).
\bibitem{ref23} R. P. Stanley. \emph{Enumerative Combinatorics}, Vol. 1. Second edition, Cambridge University Press, Cambridge (2012).
\bibitem{ref12} J. Thévenaz; P. J. Webb. Homotopy equivalence of posets with a group action. \emph{Journal of Combinatorial Theory, Series A}, 56 (1991), 173--181; \href{https://doi.org/10.1016/0097-3165(91)90030-K}{doi:10.1016/0097-3165(91)90030-K}.
\end{thebibliography}
\end{document}